\documentclass{amsart}
\usepackage{graphicx}
\usepackage{amssymb}
\usepackage{extpfeil}
\usepackage{mathtools}
\newtheorem{thm}{Theorem}[section]

\theoremstyle{definition}
\newtheorem{defn}[thm]{Definition}
\theoremstyle{remark}

\numberwithin{equation}{section}

\begin{document}
\title[Deferred Norlund statistical convergence]{Deferred Norlund statistical convergence in probability, mean and distribution for sequences of random variables}%
\author{ Kuldip Raj, Swati Jasrotia}%
\address{School of  Mathematics
Shri Mata Vaishno Devi University, Katra-182320, J \& K (India)}%
\email{kuldipraj68@gmail.com}%
\email{swatijasrotia12@gmail.com}%
\author[Saadati]{Reza Saadati$^{*}$}
\address{Reza Saadati \newline\indent  School of Mathematics, Iran University of Science and Technology, Narmak, Tehran, Iran}
\email{rsaadati@eml.cc, rsaadati@iust.ac.ir}
\thanks{}%
\subjclass[2010]{40A05, 40A30}%
\keywords{Probability convergence, deferred N\"{o}rlund, mean convergence, distribution convergence, statistical convergence}%
\begin{abstract}
We introduce and study deferred N\"{o}rlund statistical convergence in probability, mean of order $r,$ distribution and study the interrelation among them. Based upon the proposed method to illustrate the findings, we presented new Korovkin type theorems for the sequence of random variables via deferred N\"{o}rlund statistically convergence and present compelling examples to demonstrate the effectiveness of the results.
\end{abstract}
\maketitle
\numberwithin{equation}{section}
\newtheorem{theorem}{Theorem}[section]
\newtheorem{lemma}[theorem]{Lemma}
\newtheorem{proposition}[theorem]{Proposition}
\newtheorem{corollary}[theorem]{Corollary}
\newtheorem*{remark}{Remark}
\section{\textbf{Introduction and Preliminaries}}
\noindent  Fast \cite{fast} and also Schoenberg \cite{schoen} studied the concept of  statistical convergence and continued by Rath--Tripathy \cite{rath} and Gadjiev--Orhan \cite{gadjiev and Orhan}.

  Suppose that $(x_{m})$ and $(y_{m})$ are the sequences of non-negative integers fulfilling\\
\begin{equation}\label{!!}
  x_{m}<y_{m},\;\forall\; m\;\in \mathbb{N}\;\;\;\;\;\;\;\;\;\;\;\;\mbox{and}\;\;\;\;\;\;\;\;\;\;\;\;\;\;\lim_{x\rightarrow\infty}y_{m}=\infty.\\
\end{equation}
Further, let $(e_{m})$ and $(g_{m})$ be two sequences of non-negative real numbers such that\\
\begin{equation}\label{**}
  \mathcal{E}_{m}=\displaystyle\sum_{n=x_{m}+1}^{y_{m}}e_{n}\;\;\;\mbox{and}\;\;\;\mathcal{F}_{m}=\displaystyle\sum_{n=x_{m}+1}^{y_{m}}g_{n}.\\
\end{equation}
The convolution of \eqref{**} is defined as\\
$$\mathcal{R}_{m} = \displaystyle\sum_{v=x_{m}+1}^{y_{m}}e_{v}g_{y_{m}-v}.$$\\
 As introduced by Srivastava et al. in \cite{sri1}, the deferred N\"{o}rlund $(DN)$ mean is defined as
$$t_{m}= \frac{1}{\mathcal{R}_{m}}\displaystyle\sum_{n=x_{m}+1}^{y_{m}}e_{y_{m}-n}g_{n}y_{n}.$$\\

 Suppose that $(x_{m})$ and $(y_{m})$ are the sequences fulfilling conditions \eqref{!!} and $(e_{m}), (g_{m})$ are sequences satisfying \eqref{**}. A sequence $(Y_{m})$ is called as deferred N\"{o}rlund statistically convergent to $Y$ if $\forall\;\varepsilon >0,$ the set
  $$\{n:n\leq \mathcal{R}_{m}, \;\mbox{and}\;  e_{y_{m}-n}g_{n}|Y_{m}-Y|\geq \varepsilon\}$$ has zero deferred N\"{o}rlund density, i.e. if\\
  $$\lim_{m\rightarrow \infty}\frac{1}{\mathcal{R}_{m}}\Big|\Big\{ n:n\leq \mathcal{R}_{m}\;\;\mbox{and}\;\; e_{y_{m}-n}g_{n}|Y_{m}-Y|\geq \varepsilon\Big\}\Big| = 0.$$\\
  We write it as $$St_{DN}\lim Y_{m}=Y.$$

Suppose that $(x_{m})$ and $(y_{m})$ are the sequences fulfilling conditions \eqref{!!} and $(e_{m}), (g_{m})$ are sequences satisfying \eqref{**}. A sequence $(Y_{m})$  is called as deferred N\"{o}rlund statistically probability (or $St_{DNP}-$) convergent to a random variable $Y,$ if $\forall\;\varepsilon >0$ and $\delta>0,$ the set\\
$$\{n:n\leq\mathcal{R}_{m}\;\;\mbox{and}\;\; e_{y_{m}-n}g_{n}P(|Y_{m}-Y|\geq \varepsilon)\geq \delta\}$$\\ has $DN-$density zero, i.e.,
\begin{eqnarray*}
 \lim_{m\rightarrow \infty}\frac{1}{\mathcal{R}_{m}}\Big|\Big\{ n:n\leq \mathcal{R}_{m}\;\;\mbox{and}\;\; e_{y_{m}-n}g_{n}P(|Y_{m}-Y|\geq \varepsilon)
  \geq \delta\Big\}\Big| = 0
\end{eqnarray*}
or\\
\begin{eqnarray*}
 \lim_{m\rightarrow \infty}\frac{1}{\mathcal{R}_{m}}\Big|\Big\{ n:n\leq \mathcal{R}_{m}\;\;\mbox{and}\;\; 1-e_{y_{m}-n}g_{n}P(|Y_{m}-Y|\leq \varepsilon)
 \geq \delta\Big\}\Big| = 0,\\
\end{eqnarray*}
and it is denoted as\\
\begin{eqnarray*}
  St_{DNP}\lim_{m\rightarrow \infty}e_{y_{m}-n}g_{n}P(|Y_{m}-Y|\geq\varepsilon) &=& 0
\end{eqnarray*}
or\\
\begin{eqnarray*}
  St_{DNP}\lim_{m\rightarrow \infty}e_{y_{m}-n}g_{n}P(|Y_{m}-Y|\leq\varepsilon) &=& 1.
\end{eqnarray*}

\section{Deferred N\"{o}rlund statistically probability convergence}
In this section, we study deferred N\"{o}rlund statistically probability convergence,
for a historical review and basic concept we refer \cite{esi}, \cite{et},  \cite{raj}, \cite{moh1}, \cite{raj2}   \cite{jena}, \cite{mursaleen}, \cite{san}, \cite{jmcs1}, \cite{jmcs2} and \cite{sri}.
\begin{thm}
  Suppose that $(Y_{m})$ and $(Z_{m})$ are sequences of random variables and consider two random variables $Y$ and $Z.$ Then the following assertions are satisfied\\
  \begin{enumerate}
    \item $St_{DNP}Y_{m}\rightarrow Y$ and $St_{DNP}Y_{m}\rightarrow Z\Rightarrow P(Y=Z)=1,$\\
    \item $St_{DNP}Y_{m}\rightarrow y \Rightarrow St_{DNP}Y^{2}_{m}= y^{2},$\\
    \item $St_{DNP}Y_{m}\rightarrow y$ and $St_{DNP}Z_{m}\rightarrow z\Rightarrow St_{DNP}Y_{m}Z_{m}\rightarrow yz,$\\
    \item $St_{DNP}Y_{m}\rightarrow y$ and $St_{DNP}Z_{m}\rightarrow z\Rightarrow St_{DNP}\frac{Y_{m}}{Z_{m}}\rightarrow \frac{y}{z}, \;z\neq 0,$\\
    \item $St_{DNP}Y_{m}\rightarrow Y$ and $St_{DNP}Z_{m}\rightarrow Z\Rightarrow St_{DNP}Y_{m}Z_{m}\rightarrow YZ,$\\
    \item if $St_{DNP}Y_{m}\rightarrow Y\;\forall\; \varepsilon, \delta>0,$ then $\exists\; a\in \mathbb{N}$ s.t. $$d(\{ n:n\leq \mathcal{R}_{m}\;\;\mbox{and}\;\; e_{y_{m}-n}g_{n}P(|Y_{m}-Y_{a}|\geq \varepsilon)\\
  \geq \delta\})=0.$$
  \end{enumerate}
\end{thm}
\begin{proof}
Let $\epsilon$ and $\delta$ be positively small real numbers. Also consider $(x_{m})$ and $(y_{m})$ are the sequences fulfilling conditions \eqref{!!} and $(e_{m}), (g_{m})$ are sequences satisfying \eqref{**}.
  \begin{enumerate}
    \item Suppose that $a\in \Big\{n:n\leq \mathcal{R}_{m}\;\;\mbox{and}\;\; e_{y_{m}-n}g_{n}P\Big(|Y_{m}-Y|\geq \frac{\varepsilon}{2}\Big)
  < \frac{\delta}{2}\Big\}\cap \Big\{n:n\leq \mathcal{R}_{m}\;\;\mbox{and}\;\; e_{y_{m}-n}g_{n}P\Big(|Y_{m}-Z|\geq \frac{\varepsilon}{2}\Big)
  < \frac{\delta}{2}\Big\}$ (as the limit density of both the sets is $1$).
  Then, $e_{y_{m}-n}g_{n}P\Big(|Y-Z|\geq \varepsilon\Big)\leq e_{y_{m}-n}g_{n}P\Big(|Y_{a}-Y|\geq \frac{\varepsilon}{2}\Big)+e_{y_{m}-n}g_{n}P\Big(|Y_{a}-Z|\geq \frac{\varepsilon}{2}\Big)<\delta.$
  It means $$P\{Y=Z\}=1.$$
\item If $St_{DNP}Y_{m}\rightarrow 0$, then $St_{DNP}Y^{2}_{m}\rightarrow 0.$ Here, we see that $a\in \{n:n\leq \mathcal{R}_{m}\;\;\mbox{and}\;\;e_{y_{m}-n}g_{n}P(|Y_{m}-0|\geq \varepsilon) > \delta\}= a\in \{n:n\leq \mathcal{R}_{m}\;\;\mbox{and}\;\;e_{y_{m}-n}g_{n}P(|Y^{2}_{m}-0|\geq \varepsilon > \delta\}.$
    Now, take $Y_{m}^{2}=(Y_{m}-y)^{2}+2y(Y_{m}-y)+y^{2}.$ Thus, $St_{DNP}Y_{m}^{2}\rightarrow y^{2}.$\\
    \item Suppose that $St_{DNP}Y_{m}\rightarrow y$ and $St_{DNP}Z_{m}\rightarrow z.$ As $St_{DNP}Y_{m}Z_{m}= St_{DNP}\frac{1}{4}\{(Y_{m}+Z_{m})^{2}-(Y_{m}-Z_{m})^{2}\}= \frac{1}{4}\{(y_{m}+z_{m})^{2}-(y_{m}-z_{m})^{2}\}=yz.$\\
\item Suppose that $R$ and $S$ be two events correspond $|Z_{m}-z|< |z|$ and $\Big|\frac{1}{Z_{m}}-\frac{1}{z}\Big|\geq\varepsilon.$ We have\\
    $$\Big|\frac{1}{Z_{m}}-\frac{1}{z}\Big|= \frac{|Z_{m}-z|}{|zZ_{m}|}=\frac{|Z_{m}-z|}{|z|\cdot|z+(Z_{m}-z)|}\leq \frac{|Z_{m}-z|}{|z|\cdot|(|z|-|Z_{m}-z|)|}.$$
If the events $R$ and $S$ occurs at same time, then\\
$$|Z_{m}-z| \geq \frac{\varepsilon|z^{2}|}{1+\varepsilon|z|}.$$
Further, let $\varepsilon_{0}=\varepsilon |z|^{2}/(1+\varepsilon|z|)$ and $A$ be the event such that $|Z_{m}-z|\geq \varepsilon_{0}.$ Thus, $$RS\subseteq A\Rightarrow P(S)\leq P(A)+P(R^{c}).$$
Thus,\\ $\Big\{n:n\leq \mathcal{R}_{m}\;\;\mbox{and}\;\; e_{y_{m}-n}g_{n}P\Big(|\frac{1}{Z_{m}}-\frac{1}{z}|\geq \varepsilon\Big)\geq \delta\Big\}\subseteq \Big\{n:n\leq \mathcal{R}_{m}\;\;\mbox{and}\;\;$
\begin{eqnarray*}
\;\;\;\;\;\;\;\;\;\;\;\;\;\;\; e_{y_{m}-n}g_{n}P\big(|Z_{m}-z|\geq \varepsilon_{0}\big)\geq \frac{\delta}{2}\Big\}\cup \Big\{n:n\leq \mathcal{R}_{m}\;\;\mbox{and}\;\; e_{y_{m}-n}g_{n}\\
P\big(|Z_{m}-z|\geq |z|\big)\geq \frac{\delta}{2}\Big\}.
\end{eqnarray*}
Therefore, $St_{DNP}\frac{1}{Z_{m}}\rightarrow \frac{1}{z}.$ Hence, we write $St_{DNP}\frac{Y_{m}}{Z_{m}}\rightarrow \frac{y}{z},\; z\neq 0.$\\
\item Suppose that $St_{DNP}Y_{m}\rightarrow Y$ and $X$ be a random variable such that $Y_{m}X\rightarrow YX.$ Since $X$ is a random variable such that $\forall \varepsilon>0, \;\exists\; \delta>0$ and $e_{y_{m}-n}g_{n}P(|X|>\delta)\leq \frac{\varepsilon}{2}.$ Next, $\forall\; \varepsilon'>0,$
    \begin{eqnarray*}
      e_{y_{m}-n}g_{n}P\big(|Y_{m}X-YX|\geq \varepsilon') &=& e_{y_{m}-n}g_{n}P\big(|Y_{m}-Y| |X|\geq \varepsilon', |Z|>\delta\big)\\
      && + e_{y_{m}-n}g_{n}P\big(|Y_{m}-Y| |X|\geq \varepsilon', |Z|\leq\delta\big)\leq \frac{\varepsilon}{2}\\
      && +e_{y_{m}-n}g_{n}P\big(|Y_{m}-Y|\geq \frac{\varepsilon'}{\delta}\big).
    \end{eqnarray*}
    Which implies, $\big\{n:n\leq \mathcal{R}_{m}\;\mbox{and}\; e_{y_{m}-n}g_{n}P\big(|Y_{m}X-YX|\geq \varepsilon'\big)\big\}\subseteq \big\{n:n\leq \mathcal{R}_{m}\;\mbox{and}\;e_{y_{m}-n}g_{n}P\big(|Y_{m}-y|\geq \frac{\varepsilon'}{\delta}\big)\geq \frac{\varepsilon}{2}\big\}.$ Therefore, $$St_{DNP}(Y_{m}-Y)(Z_{m}-Z)\rightarrow 0.$$ Thus, $$St_{DNP}Y_{m}Z_{m}\rightarrow YZ.$$
    \item Suppose that $(x_{m})$ and $(y_{m})$ be two non-negative sequences such that\\ $$e_{y_{m}-n}g_{n}P\Big(|Y_{m}-Y|\geq \frac{\varepsilon}{2}\Big)<\frac{\delta}{2}$$ and $$\Big\{n:n\leq \mathcal{R}_{m}\;\mbox{and}\; e_{y_{m}-n}g_{n}P\Big(|Y_{m}-Y|\geq \frac{\varepsilon}{2}\Big)<\frac{\delta}{2}\Big\}=1.$$\\ Now,\\ $\big\{n:n\leq \mathcal{R}_{m}\;\mbox{and}\;e_{y_{m}-n}g_{n}P\big(|Y_{m}-Y|\geq \varepsilon\big)\geq\delta\big\}\subseteq
         \big\{n:n\leq \mathcal{R}_{m}\;\mbox{and}\; e_{y_{m}-n}g_{n}P\big(|Y_{m}-Y|\geq \frac{\varepsilon}{2}\big)<\frac{\delta}{2}\big\}=1.$
        Which implies that $$d(\{ n:n\leq \mathcal{R}_{m}\;\mbox{and}\; e_{y_{m}-n}g_{n}P(|Y_{m}-Y|\geq \varepsilon)
  \geq \delta\})=0.$$
\end{enumerate}
\end{proof}
\begin{thm}
  Suppose that $f:\mathbb{R}\rightarrow \mathbb{R}$ is uniform continuous on $\mathbb{R}$ and $St_{DNP}Y_{m}\rightarrow Y.$ Then $St_{DNP}f(Y_{m})\rightarrow f(Y).$
\end{thm}
\begin{proof}
  Let us consider a random variable $Y$ such that for each $\delta>0,\;\exists\;\beta \in \mathbb{R}$ such that $P(Y>\beta)\leq \delta/2.$ Since, $f$ is uniformly continuous on $[\beta, \beta]\;\forall \varepsilon>0,\;\exists \delta_{0}$ such that $$|f(y_{m})-f(y)|<\varepsilon\;\mbox{whenever}\; |y_{m}-y|<\delta_{0}.$$
  Thus,
  \begin{eqnarray*}
    P(|f(Y_{m})-f(Y)|\geq \varepsilon) &\leq& P(|Y_{m}-Y|\geq \delta_{0})+ P(|Y>\beta|)\\
    &\leq& P(|Y_{m}-Y|\geq \delta_{0})+ \delta/2.
  \end{eqnarray*}
  However, from the definition of $St_{DNP}-$convergence, we have\\
$$\Big\{n:n\leq \mathcal{R}_{m}\;\mbox{and}\; e_{y_{m}-n}g_{n}P\Big(|f(Y_{m})-f(Y)|\geq \varepsilon\Big)\geq \delta\Big\}$$
  \begin{eqnarray*}
    \;\;\;\;\;\;\;\;\;\;\;\;\;\;\;\;\;\;\;\;\;\;\;\subseteq \Big\{n:n\leq \mathcal{R}_{m}\;\mbox{and}\; e_{y_{m}-n}g_{n}P\Big(|Y_{m}-Y|\geq \delta_{0}\Big)<\frac{\delta}{2}\Big\}.
  \end{eqnarray*}

\end{proof}

\section{Deferred N\"{o}rlund statistical mean convergence}
\begin{defn}
  Suppose that $r\geq 1$ be a fixed number. A sequence $(Y_{m})$  is $r^{th}$ mean convergent to $Y,$ if $$\lim_{m\rightarrow \infty}E(|Y_{m}-Y|^{r})=0.$$
\end{defn}

\begin{defn}
  A sequence $(Y_{m})$  is statistically $r^{th}$ mean convergent $(MC)$ to a random variable $Y,$ where $Y:S\rightarrow\mathbb{R}$ if, $$\lim_{m\rightarrow \infty}\frac{1}{m}\Big|n:n\leq m \;\;\mbox{and}\;\; E(|Y_{m}-Y|^{r}\geq \varepsilon)\Big|=0$$ for any $\varepsilon>0.$\\
  We write it as $$St_{MC}\lim_{m\rightarrow\infty}E(|Y_{m}-Y|^{r})=0.$$
\end{defn}
\begin{defn}
  Suppose that $(x_{m})$ and $(y_{m})$ are the sequences fulfilling conditions \eqref{!!} and $(e_{m}), (g_{m})$ are sequences satisfying \eqref{**}. A sequence $(Y_{m})$  is said to be deferred N\"{o}rlund  statistically $r^{th}\; (r\geq 1)$ mean convergent to $Y$ ($Y:S\rightarrow \mathbb{R}),$ if for $\varepsilon>0,$
  \begin{eqnarray*}
    \lim_{m\rightarrow\infty}\frac{1}{\mathcal{R}_{m}}\Big|\Big\{n:n\leq \mathcal{R}_{m}\;\;\mbox{and}\;\; e_{y_{m}-n}g_{n}E\big(|Y_{m}-Y|^{r}\geq\varepsilon\big)\Big\}\Big| = 0.
  \end{eqnarray*}
  It is denoted as $$St_{DNM}\lim_{m\rightarrow\infty}E(|Y_{m}-Y|^{r})=0.$$
\end{defn}
\begin{thm}
  Let $St_{DNM}\lim_{m\rightarrow\infty}E(|Y_{m}-Y|^{r})=0\;\mbox{for}\; r\geq 1,$ then $St_{DNM}\lim_{m\rightarrow\infty}P(|Y_{m}-Y|\geq \varepsilon)=0.$
\end{thm}
\begin{proof}
  For every $\varepsilon>0,$ we have from Markov's inequality
  \begin{eqnarray*}
    St_{DNM}\lim_{m\rightarrow\infty}P(|Y_{m}-Y|\geq \varepsilon) &=& St_{DNM}\lim_{m\rightarrow\infty}P(|Y_{m}-Y|^{r}\geq \varepsilon^{r})\;\;(r\geq 1)\\
    &\leq& St_{DNM}\lim_{m\rightarrow\infty}\frac{E(|Y_{m}-Y|^{r})}{\varepsilon^{r}}=0.
  \end{eqnarray*}
  From definition of statistically deferred N\"{o}rlund mean convergence $$St_{DNM}\lim_{m\rightarrow\infty}E(|Y_{m}-Y|^{r})=0,$$
  it implies that $$St_{DNP}\lim_{m\rightarrow\infty}P(|Y_{m}-Y|\geq \varepsilon)=0.$$
\end{proof}
\noindent We now present an example to show that a sequence of random variables is statistically probability convergent but not statistically $r^{th}-$mean convergent.\\\\
\noindent\textbf{Example 1:} Suppose that $x_{m}=2m-1, y_{m}=4m-1.$ Also, suppose that $e_{y_{m}-m}=2m$ and $g_{m}=1.$ Further, consider a sequence $(z_{m})$ of random variables such that
$$Y_{m}=\left\{
  \begin{array}{ll}
m,\;\mbox{with probability}\;\frac{1}{\sqrt{m}}\\\\
0, \;\mbox{with probability}\; 1- \frac{1}{\sqrt{m}}.\\
  \end{array}
\right.$$
Then the statistically deferred N\"{o}rlund convergence of $Y_{m}$ is given as
\begin{eqnarray*}
 \lim_{m\rightarrow\infty}\frac{1}{2m}\Big|\Big\{n:n\leq \mathcal{R}_{m}\;\;\mbox{and}\;\;2mP(|Y_{m}-0|\geq \varepsilon)\Big\}\Big| &=&
 \lim_{m\rightarrow\infty}P(Y_{m}=m)\\
 &=& \lim_{m\rightarrow\infty} \frac{1}{\sqrt{m}}\\
 &=&0.
\end{eqnarray*}
However, statistically deferred N\"{o}rlund mean convergence, for $r\geq 1,$ is
\begin{eqnarray*}
  \lim_{m\rightarrow\infty}\frac{1}{2m}\Big|\Big\{n:n\leq \mathcal{R}_{m}\;\;\mbox{and}\;\;2mE(|Y_{m}-0|^{r})\Big\}\Big| &=&
  \lim_{m\rightarrow\infty}\Big(m^{r}\Big(\frac{1}{\sqrt{m}}\Big)+0\Big(1-\frac{1}{\sqrt{m}}\Big)\Big)\\
 &=& \lim_{m\rightarrow\infty} m^{r-1/2}\\
 &=& \infty.
\end{eqnarray*}
\noindent This implies that the sequence $(Y_{m})$ is $St_{DNP}-$convergent but not $St_{DNM}-$convergent.\\
\section{Statistical distribution convergence via Deferred N\"{o}rlund }
\begin{defn}
   The sequence of random variables $(Y_{m})$ is said to be distribution convergent (or convergent in distribution) to $Y,$ if $$\lim_{m\rightarrow\infty}F_{Y_{m}}(y)=F_{Y}(y)$$ for all $y\in \mathbb{R}$ at which $F_{Y}(y)$ is continuous.
\end{defn}
\noindent Thorughout the paper $(F_{Y_{m}}(y))$ is the sequence of distribution functions of $(Y_{m})$ and $F_{Y}(y)$ is the distribution function of $Y.$
\begin{defn}
The sequence $(F_{Y_{m}}(y))$ is called as statistically distribution convergent $(or \;St_{DC}),$ if there exists $F_{Y}(y)$ of random variable $Y$ such that for each $\varepsilon>0,$
\begin{eqnarray*}
  \lim_{m\rightarrow\infty}\frac{1}{m}\Big|\Big\{n:n\leq m\;\;\mbox{and}\;\; |F_{Y_{m}}(y)- F_{Y}(y)|\geq \varepsilon\Big\}\Big| = 0.
\end{eqnarray*}
We may write this as $$St_{DC}\lim_{m\rightarrow\infty}F_{Y_{m}}(y)=F_{Y}(y).$$
\end{defn}
\begin{defn}
   The sequence $(F_{Y_{m}}(y))$ of distribution functions is called as deferred N\"{o}rlund statistically distribution convergent $( or\;St_{DNDC}),$ if there exists $F_{Y}(y)$ of $Y$ such that for each $\varepsilon>0,$
\begin{eqnarray*}
  \lim_{m\rightarrow\infty}\frac{1}{\mathcal{R}_{m}}\Big|\Big\{n:n\leq \mathcal{R}_{m}\;\;\mbox{and}\;\; e_{y_{m}-n}g_{n}|F_{Y_{m}}(y)- F_{Y}(y)|\geq \varepsilon\Big\}\Big| = 0.
\end{eqnarray*}
In this case, we say
\begin{eqnarray*}
  St_{DNDC}\lim_{m\rightarrow\infty}F_{Y_{m}}(y) &=& F_{Y}(y).
\end{eqnarray*}
\end{defn}

\begin{thm}
  Suppose that $St_{DNP}\lim_{m\rightarrow\infty}P(|Y_{m}-Y|\geq \varepsilon)=0,$ then $$St_{DNDC}\lim_{m\rightarrow\infty}F_{Y_{m}}(y) = F_{Y}(y).$$
\end{thm}
\begin{proof}
  Suppose that $(F_{Y_{m}}(y))$ is distribution functions of $(Y_{m}),$ and $F_{Y}(y)$ be the distribution function of $Y.$ For $i,j\in \mathbb{R}$ such that $i<j,$ we have $$(Y\leq i)= (Y_{m}\leq j, Y\leq i)+ (Y_{m}\geq j, Y\leq i).$$
  Further, $$(Y_{m}\leq j, Y\leq i)\subseteq (Y_{m}\leq j),$$ which implies that
  \begin{equation}\label{11}
    (Y\leq i)\subseteq (Y_{m}\leq j)+(Y_{m}\geq j,Y\leq i).
  \end{equation}
  Let us take the probability to left hand side and right hand side of equation \eqref{11}
  \begin{eqnarray*}
    P(Y\leq i) &\leq& P\{(Y_{m}\leq j)+ (Y_{m}\geq j,Y\leq i)\}\\
    &\leq& P(Y_{m}\leq j)+ P(Y_{m}\geq j,Y\leq i).
  \end{eqnarray*}
  It means that
  \begin{equation}\label{22}
    F_{Y_{m}}(j)\geq F_{Y}(i)-P(Y_{m}\geq j, Y\leq i).
  \end{equation}
  If $Y_{m}\geq j, Y\leq i,$ then $Y_{m}\geq j, -Y\geq -i,$ so that $Y_{m}-Y>j-i,$ that is, $$(Y_{m}\geq j,Y\leq i)\subseteq (Y_{m}-Y>j-i)\subseteq (|Y_{m}-Y|>j-i).$$
  This means $$P(Y_{m}\geq j,Y\leq i)\leq P(|Y_{m}-Y|>j-i).$$
  As we know that $i<j$ and $St_{DNP}Y_{m}\rightarrow Y,$ we obtain $$St_{DNP}\lim_{m\rightarrow\infty}P(Y_{m}\geq j,Y\leq i)=0.$$ From \eqref{22} we get $$St_{DNDC}\lim_{m\rightarrow\infty}F_{Y_{m}}(j)\geq F_{Y}(i).$$
Similarly, if $j<a$ for any real constant $a,$ then $$(Y\leq j)= (Y\leq a, Y_{m}\leq j)+ (Y>a, Y_{m}\leq j).$$
 Consequently, $$F_{Y_{m}}(j)\leq F_{Y}(a)+ P(Y>a, Y_{m}\leq j)$$ and $$St_{DNDC}\lim_{m\rightarrow\infty}P(Y>a, Y_{m}\leq j)=0.$$
 Therefore, we get $$St_{DNDC}\lim_{m\rightarrow\infty}F_{Y_{m}}(j)\leq F_{Y}(a).$$ Thus, with $i<j<a,$ we have $$St_{DNDC}\lim_{m\rightarrow\infty}F_{Y_{m}}(j)=F_{Y}(i).$$
 \end{proof}
 \noindent\textbf{Example 2:}  Consider the random variables $((Y_{m}),Y)$ of two dimensions as $\{(0,0), (0,1),\\
  (1,0), (1,1)\}$ such that\\
 $$(Y_{m},Y)=\left\{
  \begin{array}{ll}
0,\;\;[P(Y_{m}=0, Y=0)=0= P(Y_{m}=1,Y=1)]\\\\
\frac{1}{2}, \;[P(Y_{m}=1, Y=0)=0= P(Y_{m}=0,Y=1)].\\\\
  \end{array}
\right.$$
 The distribution function of $Y_{m}$ is given by $Y_{m}=(\lambda_{1}=0,1),$ with probability mass function $$(p_{y_{m},\lambda_{1}})=P(Y_{m}=\lambda_{1}),\; \mbox{where}\; p_{y_{m},0}= \frac{1}{2}= p_{y_{m},1}$$ and for $Y= \lambda_{2}(\lambda_{2} =0,1),$  with probability mass function $$(p_{y_{m},\lambda_{2}})=P(Y_{m}=\lambda_{2}),\; \mbox{where}\; p_{y,0}= \frac{1}{2}= p_{y,1}.$$
 If $(F_{Y_{m}}(y))$ is distribution functions of $(Y_{m})$ and $F_{Y}(y)$ is the distribution function of $Y,$ then
 $$F_{Y}(y)= \lim_{m\rightarrow\infty}F_{Y_{m}}(y)=\left\{
  \begin{array}{ll}
0,\;\;(y<0)\\\\
\frac{1}{2}, \;\;(0\leq y< 1).\\\\
1,\;\; (z>1).
  \end{array}
\right.$$

\noindent Thus, we get
$$St_{DNDC}\lim_{m\rightarrow\infty}F_{Y_{m}}(y)= F_{Y}(y),\; \mbox{where}\;x_{m}=2m-1, y_{m}=4m-1, e_{y_{m}-m}=2m \;\mbox{and}\; g_{m}=1.$$ But, it is not $St_{PC}$ for the sequence of random variables, i.e. $$St_{DNPC}\lim_{m\rightarrow\infty}P(|Y_{m}-Y|\geq \varepsilon)\neq 0, \; \mbox{where}\;x_{m}=2m-1, y_{m}=4m-1, e_{y_{m}-m}=2m\;\mbox{and}\; g_{m}=1.$$

\section{Applications}
\noindent The hypothesis of the Korovkin-type theorems have been studied by several researchers in various field in different ways such as in, summability theory,  functional analysis and probability theory. Korovkin-type approximation theorems have been investigated by many mathematicians under various background, involving function spaces, Banach spaces, and so on. Recently, Mohiuddine and Alamri studied Korovkin and Voronovskaya type approximation theorems in \cite{moh}. Further, Hazarika et al. \cite{bipan} studied Korovkin approximation theorem for Bernstein operator of rough statistical convergence of triple sequences. For detailed study on Korovkin approximation theorem one may refer \cite{dutta}, \cite{moh2}, \cite{mur}, \cite{raj1}, \cite{Tripathy}.\\
\noindent By $\mathcal{C}(Y),$ we denote the space of all continuous probability functions defined on a compact subset $Z\subset \mathbb{R}.$ The space $\mathcal{C}(Y)$ is a Banach space with respect to the norm \\
$$||f||_{\infty}=\displaystyle\sup_{z\in Y}\{|f(z)|\},\;\;\;\;\;f\in \mathcal{C}(Y).$$

\noindent We say that $\mathcal{Y}$ is a positive linear operator of sequence of random variables if $$\mathcal{Y}(f,z)\geq 0 \;\mbox{whenever}\; f\geq 0.$$ Throughout, $\mathcal{Y}_{n}: \mathcal{C}(Y)\rightarrow \mathcal{C}(Y)$ be a sequence of random variables of positive linear operators.
\begin{thm}
 \cite{sri2} Let $\mathcal{Y}_{n}: \mathcal{C}(Y)\rightarrow \mathcal{C}(Y)$. Then for all $f\in \mathcal{C}(Y),$ we have\\
  \begin{equation}\label{1}
    St_{DNP}\lim_{n\rightarrow \infty}||\mathcal{Y}_{n}(f,z)-f(z)||_{\infty}=0
  \end{equation}
   iff
 \begin{equation}\label{2}
    St_{DNP}\lim_{n\rightarrow \infty}||\mathcal{Y}_{n}(1,z)-1||_{\infty}=0,
  \end{equation}
  \begin{equation}\label{3}
    St_{DNP}\lim_{n\rightarrow \infty}||\mathcal{Y}_{n}(z,z)-z||_{\infty}=0,
  \end{equation}
  \begin{equation}\label{4}
    St_{DNP}\lim_{n\rightarrow \infty}||\mathcal{Y}_{n}(z^{2},z)-z^{2}||_{\infty}=0.\\
  \end{equation}
\end{thm}

\begin{thm}
  Let $\mathcal{Y}_{n}: \mathcal{C}(Y)\rightarrow \mathcal{C}(Y)$. Then for all $f\in \mathcal{C}(Y),$ we have\\
  \begin{equation}\label{1}
    St_{DNM}\lim_{n\rightarrow \infty}||\mathcal{Y}_{n}(f,z)-f(z)||_{\infty}=0
  \end{equation}
   iff
 \begin{equation}\label{2}
    St_{DNM}\lim_{n\rightarrow \infty}||\mathcal{Y}_{n}(1,z)-1||_{\infty}=0,
  \end{equation}
  \begin{equation}\label{3}
    St_{DNM}\lim_{n\rightarrow \infty}||\mathcal{Y}_{n}(z,z)-z||_{\infty}=0,
  \end{equation}
  \begin{equation}\label{4}
    St_{DNM}\lim_{n\rightarrow \infty}||\mathcal{Y}_{n}(z^{2},z)-z^{2}||_{\infty}=0.\\
  \end{equation}\\
\end{thm}

\begin{thm}
  Let $\mathcal{Y}_{n}: \mathcal{C}(Y)\rightarrow \mathcal{C}(Y)$. Then for all $f\in \mathcal{C}(Y),$ we have\\
  \begin{equation}\label{1}
    St_{DNDC}\lim_{n\rightarrow \infty}||\mathcal{Y}_{n}(f,z)-f(z)||_{\infty}=0
  \end{equation}
   iff
 \begin{equation}\label{2}
    St_{DNDC}\lim_{n\rightarrow \infty}||\mathcal{Y}_{n}(1,z)-1||_{\infty}=0,
  \end{equation}
  \begin{equation}\label{3}
    St_{DNDC}\lim_{n\rightarrow \infty}||\mathcal{Y}_{n}(z,z)-z||_{\infty}=0,
  \end{equation}
  \begin{equation}\label{4}
    St_{DNDC}\lim_{n\rightarrow \infty}||\mathcal{Y}_{n}(z^{2},z)-z^{2}||_{\infty}=0.\\
  \end{equation}
\end{thm}
\noindent\textbf{Example 3:}
  Let $\mathcal{M}_{m}(f,y)$ be a Meyer-K\"{o}nig and Zeller operators on $\mathcal{C}[0,1]$ and $Z=[0,1]$ as defined in \cite{altin} as\\
  $$\mathcal{M}_{m}(f,y)=(1-y)^{m+1}\sum_{t=0}^{\infty}f\Big(\frac{t}{t+m+1}\Big)\binom{m+t}{t}y^{t}.$$\\
  Further, let us consider a sequence of operators $\mathcal{Y}_{n}:\mathcal{C}[0,1]\rightarrow \mathcal{C}[0,1]$ and $(Y_{n})$ as defined in example $2.4$ such that\\
  \begin{equation}\label{66}
    \mathcal{Y}_{n}(f,y)=[1+F_{Y_{m}}(y)]\mathcal{M}_{n}(f),\;\;\;(f\in \mathcal{C}[0,1]),\\
  \end{equation}
where $(F_{Y_{m}}(y))$ is defined in Example 2. Now we observe that\\
  $$\mathcal{Y}_{n}(1,y)=[1+F_{Y_{m}}(y)]\cdot1=[1+F_{Y_{m}}(y)],$$  $$\mathcal{Y}_{n}(u,z)=[1+F_{Y_{m}}(y)]\cdot y=[1+F_{Y_{m}}(y)]\cdot y$$ and
  \begin{eqnarray*}
    \mathcal{Y}_{n}(u^{2},z)&=&[1+F_{Y_{m}}(y)]\cdot \Big\{y^{2}\Big(\frac{m+2}{m+1}\Big)+\frac{y}{m+1}\Big\}.\\
  \end{eqnarray*}
  Therefore, we have
  \begin{eqnarray*}
    St_{DNDC}\lim_{n\rightarrow \infty}||\mathcal{Y}_{n}(1,y)-1||_{\infty}=0,
  \end{eqnarray*}
   \begin{eqnarray*}
   St_{DNDC}\lim_{n\rightarrow \infty}||\mathcal{Y}_{n}(y,y)-y||_{\infty}=0,
  \end{eqnarray*}
   \begin{eqnarray*}
    St_{DNDC}\lim_{n\rightarrow \infty}||\mathcal{Y}_{n}(y^{2},y)-y^{2}||_{\infty}=0.\\
  \end{eqnarray*}
Hence, $\mathcal{Y}_{n}(f,y)$ fulfills \eqref{2}, \eqref{3} and \eqref{4}. Thus, from Theorem 5.3 $$St_{DNDC}\lim_{n\rightarrow \infty}||\mathcal{Y}_{n}(f,y)-f||_{\infty}=0.$$\\
Hence, it is $(DNDC)-$convergent. However, $(Y_{m})$ is neither $(DN)-$statistical convergent nor $(DN)-$convergent. Thus, we can exhibit that the work in \cite{sri1} does not hold for our operators described in \eqref{66}. Hence, our Theorem 5.3 is stronger than the theorem proved in \cite{sri1}.\\
\section{Conclusion}
\noindent Upon prior analysis, our interest is to modify the studies of Srivastava et al. \cite{sri4} and introduce various aspects of statistical convergence for the sequences of random variables and sequences of real numbers via deferred Norlund summability mean. We first study various results presenting the connection by using fundamental limit concepts of sequences of random variables. As an applications of our findings, we present new Krorvkin-type approximation results and also demonstrated the effectiveness of the findings. As a future work one can obtain the corresponding results of the present paper using deferred Euler summability mean.

\noindent \textbf{Compliance with ethical standards}\\\\
\noindent \textbf{Availability of data and material:}
Not applicable.\\\\
\noindent \textbf{Conflict of interest:}
The authors declare that they have no conflict of interest.\\\\
\noindent \textbf{Ethical approval:}
 This article does not contain any studies with human participants or animals performed by any of the authors.\\\\
 \noindent \textbf{Acknowledgement:}  The corresponding author thanks the Council of Scientific and Industrial Research (CSIR), India for partial support under
Grant No. 25(0288)/18/EMR-II, dated 24/05/2018.\\\\

\end{document}